\documentclass[UTF-8,reqno,12pt]{amsart}
\usepackage{enumerate}
\usepackage{mhequ}

\usepackage{amssymb}

\usepackage{dutchcal}

\usepackage{bbm}
\usepackage{geometry}

\usepackage{enumitem}  
\newlist{hypothesis}{enumerate}{1}
\setlist[hypothesis]{label=({\bfseries A}\arabic*),leftmargin=*}  

\usepackage{amssymb,url,color, booktabs,nccmath}
\usepackage{mathrsfs}
\usepackage{graphicx}
\usepackage{tikz}
\usepackage{amsthm}
\usepackage{comment}

\usepackage{relsize}

\usepackage{color}
\usepackage[colorlinks=true]{hyperref}
\hypersetup{
	linkcolor=blue,          
	citecolor=red,        
	filecolor=blue,      
	urlcolor=cyan
}

\definecolor{darkergreen}{rgb}{0.0, 0.5, 0.0}

\makeatletter
		\renewcommand{\subsection}{\@startsection
			{subsection} 
			{2} 
			{0mm} 
			{0.5\baselineskip} 
			{0.3\baselineskip} 
			{\normalfont\normalsize\raggedright}} 
\makeatother
  
\allowdisplaybreaks[4]

\usepackage{fancyhdr} 
\numberwithin{equation}{section}

\newcommand{\be}{\begin{eqnarray}}
\newcommand{\ee}{\end{eqnarray}}
\newcommand{\ce}{\begin{eqnarray*}}
\newcommand{\de}{\end{eqnarray*}}
\newtheorem{theorem}{Theorem}[section]
\newtheorem{lemma}[theorem]{Lemma}
\newtheorem{remark}[theorem]{Remark}

\newtheorem{example}[theorem]{Example}

\newtheorem*{theorem*}{Theorem}
\newtheorem*{remark*}{Remark}

\def\geq{\geqslant}
\def\leq{\leqslant}

 \def\R{\mathbb R}
 \def\R{\mathbb R}    
\def\N{\mathbb N}  
   
\def\<{\langle} \def\>{\rangle}

\allowdisplaybreaks

\def\E{\mathbb{E}}
\def\P{\mathbb{P}}
\def\R{\mathbb{R}}

\def\d{\mathrm{d}}
\def\N{\mathbb{N}}

\def\bC{\mathbf{C}}

\def\1{\mathbbm{1}}
\def\e{\mathrm{e}}
\def\levy{L\'{e}vy}
\def\ito{It$\hat{\mathrm{o}}$}
\def\gron{Gr$\ddot{\mathrm {o}}$nwall}

\begin{document}

	\title{Weak solution for  
 distribution dependent SDEs driven by   L\'{e}vy noise}

	\author{Mingkun Ye}

    \address{School of Mathematics, Sun Yat-sen University, Guangzhou 510275, China}
	\email{mingkunye@foxmail.com(M. Ye)}	
	
	\begin{abstract}
 	In this paper, we establish the existence of weak solutions for distribution-dependent stochastic differential equations (DDSDEs)  driven by a broad class of L\'{e}vy noises, where the drift coefficients satisfy specific integrability conditions. This is achieved through  the Krylov-type estimate and tightness argument. 
    \\ 
	\end{abstract}

	
	\date{\today}
	
	\maketitle

\setcounter{tocdepth}{2}
\tableofcontents

\section{Introduction}
There exists a lot  of results concerning SDEs with singular coefficients. Among these, notable findings pertain to equations driven by Brownian motion. It is noteworthy that significant progress has been made for the following SDE under integrable conditions:
\begin{equation}\label{EQ:1026:01}
    \d X_{t}=b\left(t, X_{t}\right) \d t+\sigma\left(t, X_{t}\right) \d W_{t}, \quad X_{0}=x_{0}, \quad t\geq 0,
\end{equation}
where $ d \geqslant 1 $, $ b:\mathbb{R}_{+} \times \mathbb{R}^{d} \rightarrow \mathbb{R}^{d} $ is a Borel measurable function satisfying for any $ T>0 $,
\begin{equation}\label{EQ:1026:02}
    \int_{0}^{T}\Big(\int_{\mathbb{R}^{d}}|b(t, x)|^{p} \d x\Big)^{{q}/{p}} \d t<\infty \quad \text { with } \quad q, p \in[2, \infty) \quad \text { and } \quad \frac{d}{p}+\frac{2}{q}<1,
\end{equation}
and $ \sigma:\mathbb{R}_{+} \times \mathbb{R}^{d} \rightarrow \mathbb{R}^{d} \times \mathbb{R}^{d} $ is a Borel measurable function, $ \left(B_{t}\right)_{t \geqslant 0} $ is a $ d $-dimensional standard Brownian motion. In their seminal work \cite{KRYLOV05PTRF}, Krylov and Röckner established the existence of a unique strong solution to \eqref{EQ:1026:01} under the assumption that $ \sigma=\mathbb{I}_{d\times d} $ and $ b $ satisfies \eqref{EQ:1026:02}. Zhang \cite{Zhang2005SPA, ZHANG11EJP} extended this result to the multiplicative noise under some non-degenerate and Sobolev conditions on the diffusion coefficient, using Zvonkin's transformation and establishing Krylov-type estimates. Recently, their works were further complemented by Xia et al. \cite{XIA2020SPA}, and Xie and Zhang \cite{Xie_Zhang_POINCARE_2020}.

For SDEs with distribution-valued drift coefficients and additive noise, such as when the drift coefficient belongs to the Kato class, which can be seen as an extension of the integrable condition, several studies have already been conducted; see Chen et al. \cite{Chen2012AoP}, Zhang and Zhao \cite{ZhangandZhao2018CMS}, Ren and Zhang \cite{RenZhang2025Bernoulli}, among others. Very recently, Veretennikov \cite{Veretennikov2024SD} established weak existence result for the degenerate DDSDE. Mishura and Veretennikov \cite{Veretennikov2020TPMS} established weak and strong existence and weak and strong uniqueness results for DDSDE under some relaxed regularity conditions.

In recent years, the study of distribution-dependent SDEs (DDSDEs for short), also known as McKean-Vlasov SDEs or mean field models, has garnered significant attention due to their wide range of applications. We present some classical results here. For Brownian motion driven DDSDEs, which general form can be formulated as follow 
\begin{equation}\label{EQ:0416:01}
    \d X_t= b(t,X_t,\mathscr{L}_{X_t})\d t+\sigma(t,X_t,\mathscr{L}_{X_t})\d W_t,
\end{equation}
where \( b: \mathbb{R}_{+} \times \mathbb{R}^{d} \times \mathcal{P}(\mathbb{R}^{d}) \rightarrow \mathbb{R}^{d} \) and \( \sigma: \mathbb{R}_{+} \times \mathbb{R}^{d} \times \mathcal{P}(\mathbb{R}^{d}) \rightarrow \mathbb{R}^{d} \otimes \mathbb{R}^{d} \) are two Borel measurable functions, \( W \) is a \( d \)-dimensional standard Brownian motion on some filtered probability space \( (\Omega, \mathcal{F},\left(\mathcal{F}_{t}\right)_{t \geq 0}, \mathbb{P}) \), \( \mathscr{L}_{X_t}:=\mathbb{P} \circ X_{t}^{-1} \) is the marginal distribution of \( X \) at time \( t \), and \( \mathcal{P}(\mathbb{R}^{d}) \) is the space of all probability measures over \((\mathbb{R}^{d}, \mathscr{B}(\mathbb{R}^{d})) \).

Wang \cite{Wang_SPA_2018} proved the well-posedness of DDSDEs under some one-side Lipschitz assumptions, utilizing a distributional iteration argument. Subsequently, Huang and Wang \cite{Huang_SPA_2019} established the wellposedness of DDSDEs with non-degenerate noise under integrability conditions on distribution-dependent coefficients.

As an extension of \cite{KRYLOV05PTRF} to SDE with distribution-dependent coefficients, Rockner and Zhang \cite{RocknerandZhang2021Bernoulli} demonstrated the strong well-posedness of the aforementioned DDSDE \eqref{EQ:0416:01} under certain integrability assumptions in the spatial variable and Lipschitz continuity in the measure variable concerning $b$ and $\sigma$. Recently, as a generalization of \cite{RocknerandZhang2021Bernoulli} in some sense, the existence of weak solutions to DDSDE \eqref{EQ:0416:01} was demonstrated in Zhao \cite{ZHAO2025PA} when $\sigma$ satisfies certain non-degeneracy and continuity assumptions, and when $b$ meets some integrability conditions, and continuity requirements in the (generalized) total variation distance. Moreover, weak uniqueness is established under additional continuity assumptions of Lipschitz type. It is worth noting that, in the result of weak solution existence, these conditions can be further relaxed to the drift coefficients belonging to a certain Kato class (see \cite[Remark 3.3]{ZHAO2025PA} for more details).

An important class of distribution-dependent equations, the Nemytskii-type DDSDE, also termed density-dependent stochastic differential equation (dDSDE for short), was first introduced in \cite[Section 2]{BARBUandROCKNER2018SMA} to provide a probabilistic representation for solutions of nonlinear Fokker-Planck equations. In a series of works, Barbu and Röckner systematically investigated the following Nemytskii-type DDSDE driven by Brownian motions (see, e.g., monograph \cite{Barbu2024}): 
$$
\mathrm{d} X_{t} = b\left(t, X_{t}, \rho_{t}\left(X_{t}\right)\right) \mathrm{d} t + \sigma\left(t, X_{t}, \rho_{t}\left(X_{t}\right)\right) \mathrm{d} W_{t}, \quad X_{0} \stackrel{d}{=} \mu_{0},
$$
where $ \sigma: \mathbb{R}_{+} \times \mathbb{R}^{d} \times \mathbb{R}_{+} \rightarrow \mathbb{R}^{d} \otimes \mathbb{R}^{d} $ is measurable, $ \rho_{t}(x) = \mathbb{P} \circ X_{t}^{-1}(\mathrm{d} x) / \mathrm{d} x $ denotes the distributional density of $ X_{t} $, and $ W $ is a standard $ d $-dimensional Brownian motion.

The study of SDEs driven by Lévy processes has also attracted considerable attention. For the SDE given by
\begin{equation}\label{EQ:0415:04}
\d X_t =b(t,X_t)\d t+\sigma(t,X_t)\d L_t,
\end{equation}
Gy\"ongy and Krylov established the existence and uniqueness of strong solutions under some one-sided Lipschitz conditions when driven by semi-martingales. Priola \cite{Priola2012OJM} proved pathwise uniqueness for SDEs driven by non-degenerate symmetric $\alpha$-stable Lévy processes with a bounded and Hölder continuous drift term. More recently, Butkovsky et al. \cite{butkovskyStrongRateConvergence2024} proved the well-posedness of SDE driven by a general Lévy process with Hölder continuous drift and $\sigma=\mathbb{I}_{d\times d}$ using the refined stochastic sewing lemma and fixed point argument.

Concerning Lévy processes  driven DDSDEs, since Brownian motion is a continuous Lévy process, it is natural to consider density-dependent SDEs driven by pure jump Lévy processes and moreover, general \levy~process. Building on these developments, researchers have begun to naturally investigate Lévy-driven distribution-dependent equations like
\begin{equation}\label{EQ:0415:03}
\mathrm{d} X_{t} = b\left(t,X_{t}, \mathscr{L}_{X_{t}}\right) \mathrm{d} t + \sigma\left(t,X_{t}, \mathscr{L}_{X_{t}}\right) \mathrm{d} L_{t}, \quad t \in [0, T],
\end{equation}
where \( T>0 \) is a fixed constant, \( \left(L_{t}\right)_{t \geq 0} \) is a Lévy process on a complete filtration probability space $ (\Omega,(\mathcal{F}_{t})_{t \in[0, T]}, \mathbb{P})$, and $ \mathscr{L}_{X_{t}} $ is the law of $ X_{t} $, and agiain for the space \( \mathcal{P}(\mathbb{R}^{d}) \) of all probability measures on \( \mathbb{R}^{d} \) equipped with the weak topology,
\[
b:[0, T] \times \mathbb{R}^{d} \times \mathcal{P}(\mathbb{R}^{d}) \rightarrow \mathbb{R}^{d}, \quad \sigma:[0, T] \times \R^d\times \mathcal{P}(\mathbb{R}^{d}) \rightarrow \mathbb{R}^{d} \otimes \mathbb{R}^{m}
\]
are measurable functions.

Huang and Yang \cite{HUANGandYANG2021} studied the well-posedness of DDSDE \eqref{EQ:0415:03} with Hölder continuous drift driven by additive $\alpha$-stable noise, employing Zvonkin-type transformations, Krylov-type estimates, and a (Picard) distributional iteration method.

Deng and Huang sequentially examined two types of DDSDEs driven by a standard $\alpha$-stable process. Specifically, in \cite{DENGandHUANG2023ARXIV}, they considered DDSDEs with distribution-free diffusions, where $L$ is an $\alpha$-stable process with $\alpha \in (1, 2)$. The primary analytical tools were Zvonkin’s transformation, a time-change technique, and a two-step fixed point argument. In \cite{DENGandHUANG2024ARXIV}, they investigated the DDSDE \eqref{EQ:0415:03}, where $L$ is an $\alpha$-stable process with $\alpha \in (1/2, 1)$. The main analytical tools included heat kernel estimates for distribution-independent stable SDEs and Banach’s fixed point theorem.

Recently, Hao et al. \cite{hao2024supercriticalmckeanvlasovsdedriven} established the well-posedness of the DDSDE \eqref{EQ:0415:03} driven by cylindrical $\alpha$-stable processes with $\alpha\in(0,1)$, using some priori estimates and Picard iteration. In their framework, the drift coefficient is Hölder continuous in the spatial variable and Lipschitz continuous with respect to the measure variable under a weighted total variation distance.

For the following Nemytskii-type DDSDE driven by an $\alpha$-stable process with Hölder drifts:
$$
\mathrm{d} X_{t}=b\left(t, X_{t}, \rho_{t}\left(X_{t}\right)\right) \mathrm{d} t+\mathrm{d} L_{t}, \quad X_{0} \stackrel{d}{=} \mu_{0},
$$
Wu and Hao established its well-posedness based on Schauder estimates in Besov spaces for nonlocal Fokker–Planck equations and Euler-Maruyama scheme in \cite{WUandHAO2023SPA}.

The paper is organized as follows.
In Section \ref{SEC:02}, we provide a concise introduction to several function spaces, including Hölder spaces, fractional Sobolev spaces, and Bessel potential spaces. We then briefly review Lévy processes and formulate  Assumption $($\hyperlink{(A)}{$\mathbf{A}$}$)$ on the driving Lévy process for the DDSDE. In addition, we establish a Krylov-type estimate for semimartingales under this framework. In Section \ref{SEC:03}, we propose Assumption $($\hyperlink{Htheta}{\(\mathbf{H}^{\theta}\)}$)$ on the drift coefficients of the DDSDE. Using the Krylov-type estimate from Section \ref{SEC:02} and employing standard tightness arguments, we prove the existence of weak solution to the DDSDE.

Throughout this paper, we use the following conventions and notations. $\langle\cdot,\cdot\rangle$ denote the inner product among $\R^d$. Define \( \mathbb{R}_{+}:=[0, \infty) \). The letter \( C=C(\cdots) \) denotes an unimportant constant, whose value may change in different places. We also use  \( A \lesssim B \) to denote \( A \leqslant c B \), for some unimportant constant \( c > 0 \).

\section{Preliminary, Assumptions on \levy~process and Krylov Type Estimate}\label{SEC:02}

\subsection{H\"older, fractional Sobolev, and Bessel-potential space}
For every $p \in[1, \infty)$, we denote by $L^{p}$ the space of all $p$-order integrable functions on $\mathbb{R}^{d}$ with the norm denoted by $\|\cdot\|_{p}$. For $p=\infty$, we set
\begin{equation*}
    \|f\|_{\infty}:=\sup _{x \in \mathbb{R}^{d}}|f(x)|.
\end{equation*}
 For $s>0$, let $\mathbf{C}^{s}(\mathbb{R}^{d})$ be the classical $s$-order Hölder space consisting of all measurable functions $f: \mathbb{R}^{d} \rightarrow \mathbb{R}$ with
\begin{equation*}
\|f\|_{\mathbf{C}^{s}}:=\sum_{j=0}^{\lfloor s \rfloor}\left\|\nabla^{j} f\right\|_{\infty}+\left[\nabla^{\lfloor s \rfloor} f\right]_{\mathbf{C}^{s-\lfloor s \rfloor}}<\infty,
\end{equation*}
where $\lfloor s \rfloor$ denotes the greatest integer not more than $s$, and
\begin{equation*}
    [f]_{\mathbf{C}^{\gamma}}:=\sup _{x \in \mathbb{R}^{d}} \frac{\|f(\cdot+x)-f(\cdot)\|_{\infty}}{|x|^{\gamma}},\quad  \gamma \in(0,1).
\end{equation*}
So, by convention, we have $\|\cdot\|_{\mathbf{C}^{0}}=\|\cdot\|_{\infty}.$

For $\beta \geq 0$ and $p \geq 1$, let $\mathbf{H}^{\beta,p}:=(I-\Delta)^{-\beta / 2}(L^{p}(\mathbb{R}^{d}))$ be the Bessel potential space with the norm
\begin{equation*}
    \|f\|_{\beta, p}:=\left\|(I-\Delta)^{\beta / 2} f\right\|_{p}.
\end{equation*}

Notice that for $\beta=m \in \mathbb{N}$, an equivalent norm of $\mathbf{H}^{\beta,p}$ is given by (cf. \cite{Triebel1978}, p. 177)
\begin{equation*}
    \|f\|_{m, p}=\sum_{i=0}^{m}\left\|\nabla^{i} f\right\|_{p}.
\end{equation*}

By Sobolev's embedding theorem, if $\beta-\frac{d}{p}>0$ is not an integer, then (cf. \cite{Triebel1978}, p. 206, (16))
\begin{equation*}
    \mathbf{H}^{\beta,p} \hookrightarrow \bC^{\beta-d / p}(\mathbb{R}^{d}).
\end{equation*}
Let $A$ and $B$ be an interpolation pair of Banach spaces. For $\theta \in[0,1]$, we use $[A, B]_{\theta}$ to denote the complex interpolation space between $A$ and $B$ (cf. \cite{Triebel1978}). We have the following reiteration relation (cf. \cite[p. 185, (11)]{Triebel1978}): for $p>1$, $\beta_{1} \neq \beta_{2}$ and $\theta \in(0,1)$,
\begin{equation*}
    \left[\mathbf{H}^{\beta_{1},p}, \mathbf{H}^{\beta_{2},p}\right]_{\theta}=\mathbf{H}^{\beta_{1}+\theta\left(\beta_{2}-\beta_{1}\right),p}.
\end{equation*}

On the other hand, for $0<\beta \neq$ integer, the fractional Sobolev space $\mathbf{W}^{\beta,p}$ is defined by the completion of $\bC_{c}^{\infty}$ under the following norm $\|f\|_{\beta, p}^{\sim}$ (cf. \cite{Triebel1978}, p. 190, (15)) 
\begin{equation*}
    \|f\|_{\beta, p}^{\sim}:=\|f\|_{p}+\sum_{k=0}^{[\beta]}\left(\int_{\mathbb{R}^{d}} \int_{\mathbb{R}^{d}} \frac{\left|\nabla^{k} f(x)-\nabla^{k} f(y)\right|^{p}}{|x-y|^{d+(\beta-[\beta]) p}} \mathrm{d} x \mathrm{d} y\right)^{1 / p}<+\infty
\end{equation*}

For $\beta=0,1,2, \ldots$, we also set $\mathbf{W}^{\beta,p}:=\mathbf{H}^{\beta,p}$. The relation between $\mathbf{H}^{\beta,p}$ and $\mathbf{W}^{\beta,p}$ is given as follows (cf. \cite{Triebel1978}, p. 180, (9)): for any $\beta>0, \varepsilon \in(0, \beta)$ and $p>1$,
\begin{equation*}
    \mathbf{H}^{\beta+\varepsilon,p} \hookrightarrow \mathbf{W}^{\beta,p} \hookrightarrow \mathbf{H}^{\beta-\varepsilon,p} .
\end{equation*}

\subsection{Assumptions on \levy~process}

A càdlàg process \((L_{t})_{t\geq 0 }\) on \( \mathbb{R}^{d}~(d \geqslant 1) \) is called a Lévy process, if \( L_{0}=0 \) almost surely and \( L \) has independent and identically distributed increments. Let $\mathbb{R}^{d}_{*}:=\mathbb{R}^{d}\setminus\{0\}$. The associated Poisson random measure of $(L_t)$ is defined by
\[
N((0, t] \times \Gamma):=\sum_{s \in(0, t]} \mathbbm{1}_{\Gamma}\left(L_{s}-L_{s-}\right), \quad \forall \Gamma \in \mathscr{B}(\mathbb{R}^{d}_{*}), t>0,
\]
and the Lévy measure  is given by
\[
\nu(\Gamma):=\mathbb{E} N((0,1] \times \Gamma) .
\]
Then, the compensated Poisson random measure is defined by
\[
\tilde{N}(\mathrm{d} r, \mathrm{d} z):=N(\mathrm{d} r, \mathrm{d} z)-\nu(\mathrm{d} z) \mathrm{d} r.
\]

Fix a probability space \( (\Omega, \mathcal{F}, \mathbb{P}) \) and on it a \( d \)-dimensional Lévy process \( L \), equipped with a right-continuous, complete filtration \( \mathbb{F}=\left(\mathcal{F}_{t}\right)_{t \geq 0} \). 
The Markov transition semigroup associated with the process \( L \) is denoted by \( \mathcal{P}=\left({P}_{t}\right)_{t \geq 0} \), and the generator of \( \mathcal{P} \) is denoted by \(\mathcal{L} \). To be concrete, 
\begin{equation}\label{EQ:SEMIGROUP:01}
	{P}_{t} f(x)=\mathbb{E}\left(f\left(L_{t}+x\right)\right).
\end{equation}
Then, suppose $\alpha\in (1,2] $, we propose assumptions (\hypertarget{(A)}{$\mathbf{A}$}) about the driving \levy~process $ L$:

\noindent \hypertarget{(A1)}{(A1)} (gradient type bound on the semigroup) For any $p\geq 1,$ there exists a constant $M$ such that for any $f\in \mathbf{C}_0^{\infty}(\R^d),$
\[
\|\nabla P_t f\|_{p}\leq Mt^{-\frac{1}{\alpha}}\|f\|_{p}.
\]



\noindent \hypertarget{(A2)}{(A2)} (moment conditions) For any $T>0,$ there exists a constant \( M=M(\alpha,T) \) such that
\[
\mathbb{E}\left[\left|L_{t}\right| \right] \leq M t^{\frac{1}{\alpha}}, \quad 0<t\leq T.
\]

\noindent \hypertarget{(A3)}{(A3)} (smooth density) For any \( t>0, \mu_{t}(\d x):=\mathscr{L}_{L_t}(\d x)=\P\circ L_t^{-1}(\d x) \) has a smooth density \( p_{t}(x) \) with respect to the Lebesgue measure $\d x$, which is given by
\[
p_{t}(x)=\frac{1}{(2 \pi)^{d}} \int_{\mathbb{R}^{d}} \mathrm{e}^{-\mathrm{i}\langle x, \xi\rangle} \mathrm{e}^{-t \Phi(\xi)} \mathrm{d} \xi,
\]
where $\Phi(\xi)$ is the \levy~symbol corresponding $L$.

\noindent \hypertarget{(A4)}{(A4)} (path decomposition) Let $\R^d_{*}:=\R^d\setminus\{0\}.$
By  \levy-\ito's decomposition, for any $t\geq 0, L_t $ can be reformulated as following
\begin{equation}\label{EQ:0722:01}
	L_t= \sigma W_t+\int_{0}^{t}\int_{0<|x|\leq 1}x\tilde{N}(\d s,\d x)+\int_{0}^{t}\int_{|x|>1}x{N}(\d s,\d x):=\sigma W_t+Z_t, \quad \sigma\in \R^d\otimes\R^d, 
\end{equation}
where $ (W_t) $ is a $ d $-dimensional standard Brownian motion, $\sigma\sigma^{\top}=:A$ is either positive-definite matrix or $0$, and $ {N}(\d s,\d x),\tilde{N}(\d s,\d x)$ are the Poisson random measure and the compensated Poisson random measure on $ \mathscr{B}([0,T])\times \mathscr{B}(\R_{*}^d) $ with respect to a pure-discontinuous process $(Z_t)$, respectively.

\noindent \hypertarget{(A5)}{(A5)} (the non-degenerate condition) The support of \levy~measure $\nu$ is not contained in a proper linear subspace of $\R^d$ when $Z\not\equiv 0$ in \eqref{EQ:0722:01}.

\begin{remark}
    If follows from \cite[Proposition 24.17]{SATO2013} that $($\hyperlink{(A5)}{$\mathrm{A}5$}$)$ is equivalent to the distribution of $L_t$ is nondegenerate.
\end{remark}

\begin{remark}
 By the L\'{e}vy--Khintchine formula, if condition $($\hyperlink{(A4)}{$\mathrm{A}4$}$)$ holds for the process $ L $, then its generating triplet is $(A, \nu, 0)$, where $\nu$ denotes the L\'{e}vy measure of $ L $. Furthermore, the characteristic function of $ L_t $ satisfies
\begin{equation*}
\mathbb{E}\left[\,\mathrm{e}^{\mathrm{i}\langle\xi, L_t\rangle}\,\right] = \mathrm{e}^{-t \Phi(\xi)},
\end{equation*}
where the L\'{e}vy symbol $\Phi(\xi)$ is given by
\begin{equation}
    \Phi(\xi) = -\frac{1}{2}\langle\xi, A\xi\rangle + \int_{\mathbb{R}^d} \left( \mathrm{e}^{\mathrm{i}\langle\xi, y\rangle} - 1 - \mathrm{i}\langle\xi, y\rangle \mathbbm{1}_{\{|y| \leq 1\}} \right) \nu(\mathrm{d}y).
\end{equation}
Note that when $($\hyperlink{(A3)}{$\mathrm{A}3$}$)$  holds, the semigroup $(P_t)_{t\geq 0}$ can be expressed by the convolution, i.e.,
\begin{equation}\label{EQ:SEMIGROUP:02}
	{P}_{t} f(x)=\int_{\mathbb{R}^{d}} p_{t}(z) f(x-z) \mathrm{d} z=\int_{\mathbb{R}^{d}} p_{t}(x-z) f(z) \mathrm{d}z=p_t*f(x).
\end{equation}
\end{remark}

%

\begin{remark}\label{RMK:1023:01}

From $($\hyperlink{(A1)}{$\mathrm{A}1$}$)$ , we can deduce the following estimate:
\[
    \|\nabla^k P_t f\|_{p} \lesssim   t^{-\frac{k}{\alpha}} \|f\|_{p}.
\]
We will proceed by induction on \( k \). For \( k = 1 \), the result holds directly by assumption.  Assume that the inequality holds for \( k-1 \), that is,
\[
    \|\nabla^{k-1} P_t f\|_{p} \lesssim  t^{-\frac{k-1}{\alpha}} \|f\|_{p}.
\]
We now show that the result holds for \( k \). Using the semigroup property of \( (P_t) \), we have
\begin{equation} \label{EQ:1023:01}
    \begin{split}
        \|\nabla^k P_t f\|_{p} &= \|\nabla \nabla^{k-1} P_t f\|_{p} = \|\nabla \nabla^{k-1} P_{\frac{t}{2}} (P_{\frac{t}{2}} f)\|_{p} \\
        &= \|\nabla^{k-1} P_{\frac{t}{2}} (\nabla P_{\frac{t}{2}} f)\|_{p} \lesssim  t^{-\frac{k-1}{\alpha}} \|\nabla P_{\frac{t}{2}} f\|_{p} \\
        &\lesssim  t^{-\frac{k-1}{\alpha}} t^{-\frac{1}{\alpha}} \|f\|_{p} =  t^{-\frac{k}{\alpha}} \|f\|_{p}.
    \end{split}
\end{equation}
Thus, the inequality holds for all \( k \geq 1 \).

\end{remark}

\begin{example}
	We can prove,   there are many kinds of process   satisfy  the above assumptions $($\hyperlink{(A1)}{$\mathrm{A}1$}$)$-$($\hyperlink{(A5)}{$\mathrm{A}5$}$)$, for example,
	\begin{enumerate}
	\item General non-degenerate $ \alpha $-stable process. Take
\begin{equation}\label{EQ:0515:01}
    \nu(D)=\int_{0}^{\infty} \int_{\mathbb{S}} r^{-1-\alpha} \mathbbm{1}_{D}(r \xi) \mu(\d \xi) \d r,  \quad D \in \mathscr{B}(\mathbb{R}^{d}),
\end{equation}
where $\mu$ is a finite non-negative measure concentrated on the unit sphere $\mathbb{S}:=\{y \in$ $\left.\mathbb{R}^{d}:|y|=1\right\}$ that is non-degenerate. 
	\item Standard isotropic d-dimensional $ \alpha $-stable process, $ \alpha \in  (1, 2) $. Let $\mu$ in \eqref{EQ:0515:01} be the uniform measure on $\mathbb{S}$. It follows that 
\begin{equation*}
    \nu(D)=c_{\alpha}\int_{D}\frac{1}{|y|^{d+\alpha}}\d y,\quad D\in \mathscr{B}(\R^d).
\end{equation*}
	\item Cylindrical $ \alpha $-stable process, $ \alpha \in  (1, 2) $, that is, take $\mu$ in \eqref{EQ:0515:01} being $\sum_{i=1}^{d}\delta_{e_i}+\delta_{-e_i}$, where $\{e_i\}_{1\leq i\leq d}$ is the standard basis in $\R^d$.
	\item $ \alpha $-stable type process, $ \alpha \in  (1, 2) $. Take 
\begin{equation}\label{EQ:0515:02}
    \nu(D)=\int_{0}^{\infty} \int_{\mathbb{S}} r^{-1-\alpha} \rho(r) \mathbbm{1}_{D}(r \xi) \mu(
\d \xi) \d r, \quad  \forall D \in \mathscr{B}(\mathbb{R}^{d}),
\end{equation}
where $\mu$ is a symmetric, non-degenerate finite and non-negative measure concentrated on $\mathbb{S}$ and $\rho:(0, \infty) \rightarrow \mathbb{R}_{+}$ is a measurable function such that for some constants $C, C_{1}, C_{2}>0$ one has
\begin{equation}\label{EQ:0515:05}
    \mathbbm{1}_{[0, C]}(r) \leq C_{1} \rho(r) \leq C_{2}.
\end{equation}
We can see that \eqref{EQ:0515:01} is the special case of \eqref{EQ:0515:02} with $\rho(r)\equiv 1$.
	\item $ \alpha $-stable tempered process, $ \alpha \in  (1, 2) $. Take $\rho(r)$ in \eqref{EQ:0515:02} as $\exp(-cr)$ for some constant $c > 0$.
	\item Truncate $ \alpha $-stable process, $ \alpha \in  (1, 2) $. Take $\rho(r)$ in \eqref{EQ:0515:02} as $c\mathbb \1_{[0,1]}(r)$ with $c>0$, and let $\mu$ be the uniform measure on $\mathbb{S}$. It yields that
    \begin{equation*}
        \nu(D)=c_{\alpha} \int_{D \cap\{y:|y| \leq 1\}}|y|^{-d-\alpha} \d y, \quad \forall D \in \mathscr{B}(\mathbb{R}^{d}).
    \end{equation*}
	\item Brownian motion;
\item The linear combination of two independent \levy~processes, i.e., $ L^{(1)} ,L^{(2)}$ are independent  \levy~processes that satisfy $($\hyperlink{(A1)}{$\mathrm{A}1$}$)$-$($\hyperlink{(A5)}{$\mathrm{A}5$}$)$ with $\alpha=\alpha_1,\alpha_2$ respectively, then so does $ L^{(1)}+L^{(2)}$ with $\alpha=\alpha_1\vee \alpha_2$. In particular, $L=L^{(1)}+L^{(2)}$  can be the sum of a $ d $–dimensional Brownian motion and the standard $ \alpha $–stable process.
\end{enumerate}
\end{example}
\begin{proof}
   For (1)-(6), we only need to prove (4) essentially, as (1)-(3), (5), (6) are the special cases of (4). Then, by the fact that $\mu$ is symmetric, one has the symbol of $L$ can be reformulated as
   \begin{equation*}
       \Phi(\lambda)=\int_{0}^{\infty} \int_{\mathbb{S}}(1-\cos (r\langle\lambda, \xi\rangle) r^{-1-\alpha} \rho(r) \mu(\d \xi) \d r\geq 0,\quad \forall \lambda \in \R^d.
   \end{equation*}
   It is easy to see that non-degeneracy of $\mu$ implies that $\int_{\mathbb{S}}|\langle\lambda, \xi\rangle| \mu(\d \xi)>0$ for any $\lambda \in \mathbb{S}$ and thus
\[
\inf _{\lambda \in \mathbb{S}} \int_{\mathbb{S}}|\langle\lambda, \xi\rangle| \mu(\d \xi)>0.
\]
Denoting $\bar{\lambda}:=\frac{\lambda}{|\lambda|} $ and applying \eqref{EQ:0515:05}, we get for any  $\lambda \in \mathbb{R}^{d} $  with $|\lambda|>1 / C$,
\begin{equation}\label{EQ:0515:06}
    \begin{split}
        \Phi(\lambda)&  \geq \int_{\mathbb{S}} \int_{0}^{C}(1-\cos (r\langle\lambda, \xi\rangle) r^{-1-\alpha} \d r \mu(\d \xi) \\
&=|\lambda|^{\alpha} \int_{\mathbb{S}}|\langle\bar{\lambda}, \xi\rangle|^{\alpha} \int_{0}^{C|\langle\lambda, \xi\rangle|}(1-\cos r) r^{-1-\alpha} \d r \mu(\d \xi) \\
&\geq \frac{1}{10}|\lambda|^{\alpha} \int_{\mathbb{S}}|\langle\bar{\lambda}, \xi\rangle|^{\alpha} \int_{0}^{|\langle\bar{\lambda}, \xi\rangle|} r^{1-\alpha} \d r \mu(\d \xi) \\
&\geq N|\lambda|^{\alpha} \int_{\mathbb{S}}|\langle\bar{\lambda}, \xi\rangle|^{2} \mu(\d \xi)  \geq N|\lambda|^{\alpha}\left(\int_{\mathbb{S}}|\langle\bar{\lambda}, \xi\rangle| \mu(\d \xi)\right)^{2}\\
&\geq N|\lambda|^{\alpha}.
    \end{split}
\end{equation}
Moreover, for any $n\in \N,$ by \eqref{EQ:0515:06},
\begin{align*}
        \int_{\R^d} |\E\exp(i\langle\lambda ,L_t\rangle)||\lambda|^n\d \lambda &= \int_{\R^d}\e^{-t\Phi(\lambda)}|\lambda|^n \d \lambda
        \\
        & = \left(\int_{|\lambda|>1/C}+\int_{|\lambda|\leq 1/C}\right)\e^{-t\Phi(\lambda)}|\lambda|^n \d \lambda\\
        &\leq \int_{|\lambda|>1/C}\e^{-N|\lambda|^{\alpha}}|\lambda|^{n}\d \lambda+ \int_{|\lambda|\leq 1/C} |\lambda|^n\d \lambda<\infty.
\end{align*}
Then by \cite[Proposition 28.1]{SATO2013}, for any $t>0$, the distribution of $L_t$ has a smooth density.

By Minkowski's inequality and the gradient estimate of heat kernel (see, e.g., \cite[(1.2)]{DuZhang2019}), one has 
\begin{equation*}
    \begin{split}
        \|\nabla P_t f\|_{p} & = \left\|\int_{\R^d}\nabla p(t,y)f(x-y)\d y\right\|_{p}\\
        &\lesssim \left\|\int_{\R^d}t^{1-1/\alpha}(t^{1/\alpha}+|y|)^{-d-\alpha}f(x-y)\d y\right\|_{p}\\
        &\lesssim \|f\|_{p}\int_{\R^d}t^{1-1/\alpha}(t^{1/\alpha}+|y|)^{-d-\alpha}\d y\\
        &=\|f\|_{p} t^{1-1/\alpha} \frac{2 \pi^{d / 2}}{\Gamma(d / 2)}  \frac{\Gamma(d) \Gamma(\alpha)}{\Gamma(d+\alpha)}  t^{-1}\lesssim t^{-1/\alpha}\|f\|_{p}.
    \end{split}
\end{equation*}

   For (7), i.e., when $L$ is a Brownian motion, the proof is obvious.

   For (8),    when $L=L^{(1)}+L^{(2)}$, $ L^1 ,L^2$ are independent  \levy~processes satisfy $($\hyperlink{(A)}{$\mathbf{A}$}$)$  with index   $\alpha_1,~\alpha_2$, respectively, then 
    \begin{equation*}
        P_tf(x)=\E \left(P_t^{(2)}f(x+L_t^{(1)})\right),
    \end{equation*}
    it follows that
    \begin{equation*}
\begin{split}
            \|\nabla P_{t}f\|_{p}&=\|\E [P_t^{(2)}f(\cdot+L_t^{(1)})]\|_{p} \leq \E[\| \nabla P_t^{(2)}f(\cdot+L_t^{(1)})\|_{p}]\\
        & \lesssim t^{-\frac{1}{\alpha_2}}\|f\|_{p}. 
\end{split}
    \end{equation*}
    In same way,
    \begin{equation*}
         \|\nabla P_{t}f\|_{p} \lesssim t^{-\frac{1}{\alpha_1}}\|f\|_{p}, 
    \end{equation*}
    which yields that $ \|\nabla P_{t}f\|_{p} \lesssim  t^{-\frac{1}{\alpha_1\vee \alpha_2}}\|f\|_{p}$.

    If $L^{(1)}_{t}$ and $L^{(2)}_{t}$  have density functions  $p_t^{(\alpha_1)}(x)$ and $p_t^{(\alpha_2)}(x)$ respectively, then according to Fubini’s theorem, it is easy to see that $L^{(1)}_{t}+L^{(2)}_{t}$ has a density function
    \begin{equation*}
        p_t(x) = p_t^{(\alpha_1)}*p_t^{(\alpha_2)}(x),
    \end{equation*}
    which is smooth function.
\end{proof}

\subsection{Krylov type estimate}


In the following, we always assume that the driving noise  $L$ satisfies Assumption $($\hyperlink{(A)}{$\mathbf{A}$}$)$.

Based on $($\hyperlink{(A1)}{$\mathrm{A}1$}$)$-$($\hyperlink{(A5)}{$\mathrm{A}5$}$)$, we now formulate some consequences about heat kernel and related bounds of the process $ ( L_t ) $.
In this paper, we will treat with  the following SDE and DDSDE,
\begin{equation}\label{EQ:SDE:101}
	\d X_t= b(t,X_t)\d t+\d L_t,
\end{equation}
and 
\begin{equation}\label{EQ:MVSDE:101}
	\d X_t= b(t,X_t,\mathscr{L}_{X_t})\d t+\d L_t.
\end{equation}

We recall the following complex interpolation theorem (cf. \cite[p. 59, Theorem (a)]{Triebel1978}).
\begin{theorem}\label{THM:INTER:01}
    Let \( A_{i} \subset B_{i}, i=0,1 \), be Banach spaces. Let \( \mathcal{T}: A_{i} \rightarrow B_{i}, i=0,1 \), be bounded linear operators. For \( \theta \in[0,1] \), we have
\[
\|\mathcal{T}\|_{A_{\theta} \rightarrow B_{\theta}} \leq\|\mathcal{T}\|_{A_{0} \rightarrow B_{0}}^{1-\theta}\|\mathcal{T}\|_{A_{1} \rightarrow B_{1}}^{\theta},
\]
where \( A_{\theta}:=\left[A_{0}, A_{1}\right]_{\theta}, B_{\theta}:=\left[B_{0}, B_{1}\right]_{\theta} \) and \( \|\mathcal{T}\|_{A_{\theta} \rightarrow B_{\theta}} \) denotes the operator norm of \( \mathcal{T} \) mapping \( A_{\theta} \) to \( B_{\theta} \).
\end{theorem}

Let \( f \) be a locally integrable function on \( \mathbb{R}^{d} \). The Hardy-Littlewood maximal function is defined by
\[
\mathcal{M} f(x):=\sup _{0<r<\infty} \frac{1}{\left|B_{r}\right|} \int_{B_{r}}|f(x+y)| \mathrm{d} y
\]
where \( B_{r}:=\left\{x \in \mathbb{R}^{d}:|x|<r\right\} \). The following well-known results can be found in \cite[p. 5, Theorem 1]{Stein1970},
\begin{lemma}
    (i) For \( f \in \mathbf{W}^{1,1} \), there exists a constant \( C_{d}>0 \) and a Lebesgue measure zero set \( E \) such that for all \( x, y \notin E \),
\begin{equation}\label{EQ:0406:01}
    |f(x)-f(y)| \leq C_{d}|x-y|(\mathcal{M}|\nabla f|(x)+\mathcal{M}|\nabla f|(y)).
\end{equation}
(ii) For \( p>1 \), there exists a constant \( C_{d, p}>0 \) such that for all \( f \in L^{p}(\mathbb{R}^{d}) \),
\begin{equation}\label{EQ:0406:02}
    \|\mathcal{M} f\|_{p} \leq C_{d, p}\|f\|_{p}.
\end{equation}
\end{lemma}

To derive the Krylov-type estimate, which will be used in prove the weak existence of SDE \eqref{EQ:SDE:101}, we need the following lemma.
\begin{lemma}\label{LEM:KEY:01}
(i) For any \( \alpha \in(1,2],~p>1 \) and \( \beta, \gamma \geq 0 \), for all \( f \in \mathbf{H}^{\beta,p} \), we have
\begin{equation}\label{EQ:0620:01}
		\left\|{P}_{t} f\right\|_{\beta+\gamma, p} \leq C t^{-\gamma / \alpha}\|f\|_{\beta, p}.
\end{equation}
(ii) For any \( \alpha \in(1,2], \theta \in[0,1] \) and \( p>1 \), there exists a constant \( C=C(d, p, \theta)>0 \), such that for all \( f \in \mathbf{H}^{\theta,p} \), we have
\begin{equation}\label{EQ:0620:02}
		\left\|{P}_{t} f-f\right\|_{p} \leq C t^{\theta / \alpha}\|f\|_{\theta, p}.
\end{equation}
\end{lemma}

\begin{proof}
	(i) Let $f\in \bC_0^{\infty}(\R^d).$ For any $k,m\in \N,$ 
	\[ 
	\nabla^{k+m} P_tf(x) = \nabla^{k+m} \int f(y)P_t(x,\d y) = \nabla^{k} \int \nabla^m f(y) P_t(x,\d y),
	\]
	 by assumption $($\hyperlink{(A1)}{$\mathrm{A}1$}$)$, Remark \ref{RMK:1023:01} and iteration, it follows that
	\[ 
	\|\nabla^{k+m}P_tf\|_p \leq Mt^{-\frac{k}{\alpha}} \|\nabla^m f\|_p.
	\]
	 Since \( \bC_0^{\infty}(\R^d) \) is dense in \( \mathbf{H}^{m,p}  \), for any \( f\in \mathbf{H}^{m,p} \), we also have
	\[ 
	\|\nabla^{k+m}P_tf\|_p \leq Mt^{-\frac{k}{\alpha}} \| f\|_{m,p}.
	\]
	 Furthermore, using the contraction property of the Markov semigroup, \( \|P_t f\|_p\leq \|f\|_p \), and applying the complex interpolation theorem, i.e., Theorem \ref{THM:INTER:01}, we obtain \eqref{EQ:0620:01}.
	
	(ii) First assume \( f \in \mathbf{H}^{1,p} \). According to \eqref{EQ:0406:01}, for Lebesgue almost every \( x \in \mathbb{R}^{d} \), we have
	\[
	\begin{aligned}
		\left|{P}_{t} f(x)-f(x)\right| & \leq \int_{\mathbb{R}^{d}}|f(x+y)-f(x)| \cdot p_{t}(y) \mathrm{d} y \\
		& \leq C \int_{\mathbb{R}^{d}}(\mathcal{M}|\nabla f|(x+y)+\mathcal{M}|\nabla f|(x)) \cdot|y| \cdot p_{t}(y) \mathrm{d} y,
	\end{aligned}
	\]
	 thus, using \eqref{EQ:0406:02} and moment condition (\hyperlink{(A2)}{A2}), we have
	\[
	\left\|{P}_{t} f-f\right\|_{p} \leq C\|\mathcal{M}|\nabla f|\|_{p} \int_{\mathbb{R}^{d}}|y| \cdot p_{t}(y) \mathrm{d} y \leq C\|\nabla f\|_{p} \mathbb{E}\left|L_{t}\right|\leq C t^{1 / \alpha}\|\nabla f\|_{p}.
	\]
	 Thus, the result follows again from Theorem \ref{THM:INTER:01}.
\end{proof}

To proceed, let $ (X_t)  $  be a generally semi-martingale (no matter whether continuous or not) with the form
\[ 
 X_t =X_0+\int_{0}^{t}\xi_s \d s+L_t,
 \]
 where $ X_{0} \in \mathcal{F}_{0}   $, $L_t:=\sigma W_t+\int_{0}^{t}\int_{|x|\leq 1}x\tilde{N}(\d s,\d x)+\int_{0}^{t}\int_{|x|> 1}x{N}(\d s,\d x)$,  $ \left(\xi_{t}\right)_{t \geq 0}  $  is a measurable and  $ (\mathcal{F}_{t}) $-adapted $ \R^d  $-valued process. Let $ u(t,x) $ be a bounded smooth function on $ \R^+\times\R^d $.
By \ito 's  formula we have
\[ 
\begin{aligned}
	u\left(t,X_{t}\right)&=  u\left(0,X_{0}\right)+\int_{0}^{t}\big[\left(\partial_{s} u(s,\cdot)+\mathcal{L} u(s,\cdot)\right)\left(X_{s}\right)+\langle \xi_s ,\nabla u(s,\cdot)\rangle (X_{s})\big] \mathrm{d} s \\
	&\quad  +\int_{0}^{t}\nabla u(s,X_s)\sigma\d B_s+\int_{0}^{t} \int_{\mathbb{R}^{d}_{*}}\big(u\left(s,X_{s-}+y\right)-u\left(s,X_{s-}\right)\big) \tilde{N}(\mathrm{d} s, \mathrm{d} y)  ,
\end{aligned}
 \]
where 
\[ 
\mathcal{L}u(t,x)= \frac{1}{2}\operatorname{tr}\left(\sigma\sigma^{\top}\nabla^2_{x} u(t,x)\right)+\int_{\R^d_{*}}\left(u(t,x+z)-u(t,x)-\1_{\{|z|\leq 1\}}\langle z,\nabla u(t,x)\rangle \right) \nu(\d z).
 \]
Recall that $\mathcal{L}$
 is actually the infinitesimal generator of the semigroup $\mathcal{P}=(P_t)_{t\geq 0}$. Let us first prove the following Krylov type estimate for the above \( (X_{t})\),  which will be used later to prove the existence of weak solutions for SDE \eqref{EQ:SDE:101} and DDSDE \eqref{EQ:MVSDE:101}  with singular drift coefficient $b$.
\begin{theorem}[Krylov type estimate]\label{THM:KRYLOV:01}
Suppose that \( \alpha \in(1,2]\), \( p>\frac{d}{\alpha-1} \) and \( q>\frac{p \alpha}{p(\alpha-1)-d} \). Then, for any \( T_{0}>0 \), there exist a constant \( C= \) \( C\left(T_{0}, d, \alpha, p, q\right)>0 \) such that for any \( \left(\mathcal{F}_{t}\right) \)-stopping time \( \tau \), and \( 0 \leq S \leq T \leq T_{0} \), and all \( f \in L^{q}([S, T] ; L^{p}(\mathbb{R}^{d})) \),
\begin{equation}\label{EQ:0715:01}
	\mathbb{E}\left(\int_{S \wedge \tau}^{T \wedge \tau} f_{S}\left(X_{S}\right) \mathrm{d} s \bigg| \mathcal{F}_{S}\right) \leq C\left[1+\mathbb{E}\left(\int_{S \wedge \tau}^{T \wedge \tau}\left|\xi_{s}\right| \mathrm{d} s \bigg| \mathcal{F}_{S}\right)\right]\|f\|_{L^{q}([S, T] ; L^{p}(\mathbb{R}^{d}))} .
\end{equation}
\end{theorem}
\begin{proof}
 Let us first assume that \( f \in \bC_{0}^{\infty}(\mathbb{R}_{+} \times \mathbb{R}^{d}) \) and define
	\[
	u(t,x)=\int_{0}^{t} P_{t-s} f(s,x) \mathrm{d} s,
	\]
	where \( P_{t} \) is defined by \eqref{EQ:SEMIGROUP:02}. By Lemma \ref{LEM:KEY:01}, it is easy to see that \( u(t,x) \in \bC^{\infty}(\mathbb{R}_{+} \times \mathbb{R}^{d}) \) and solves the following integro-partial differential equation (IPDE to be short):
	\[
	\partial_{t} u(t,x)=\mathcal{L} u(t,x)+f(t,x) .
	\] 
Choosing \( \gamma \in (1+\frac{d}{p}, \alpha-\frac{\alpha}{q}) \), by  \eqref{EQ:0620:01} and Hölder's inequality, we have
\begin{equation}\label{EQ:0715:02}
	\begin{aligned}
		\left\|u(t,\cdot)\right\|_{\gamma, p} & \leq \int_{0}^{t}\left\|P_{t-s} f(s,\cdot)\right\|_{\gamma, p} \mathrm{~d} s \leq C \int_{0}^{t}(t-s)^{-\gamma / \alpha}\left\|f(s,\cdot)\right\|_{p} \mathrm{d} s \\
		& \leq C\left(\int_{0}^{t}(t-s)^{-q^{*} \gamma / \alpha} \mathrm{d} s\right)^{1 / q^{*}}\|f\|_{L^{q}(\mathbb{R}_{+} ; L^{p}(\R^d))} \leq C\|f\|_{L^{q}(\mathbb{R}_{+}; L^{p}(\R^d))},
	\end{aligned}
\end{equation}
	where \( q^{*}=q /(q-1) \).
 Let \( T_{0} > 0 \) be fixed, and let \( \tau \) be a stopping time with respect to the filtration \( \mathbb{F} \). Applying \ito's formula to \( u(T_{0} - t, X_{t}) \) and utilizing Doob's optional stopping theorem, 
	\[
\begin{aligned}
		&\quad \mathbb{E} \left(u\left(T_{0}-T \wedge \tau,X_{T \wedge \tau}\right)|\mathcal{F}_{S}\right)-u\left(T_{0}-S \wedge \tau,X_{S \wedge \tau}\right) \\
	& =\mathbb{E}\left(\int_{S \wedge \tau}^{T \wedge \tau}\left[\partial_{s} u(T_{0}-s,\cdot)+\mathcal{L} u(T_{0}-s,\cdot)\right]\left(X_{s}\right)+\langle \xi_{s}, \nabla u(T_{0}-s,\cdot)\left(X_{s}\right)\rangle \mathrm{d} s \middle|\mathcal{F}_{S}\right) \\
	& \leq \mathbb{E}\left(\int_{S \wedge \tau}^{T \wedge \tau}\left[-f\left(s,X_{s}\right)+\left|\xi_{s}\right| \cdot\left|\nabla u({T_{0}-s},X_{s})\right|\right] \mathrm{d} s \middle| \mathcal{F}_{S}\right),
\end{aligned}
	\]
	which yields by \eqref{EQ:0715:02} and Sobolev’s embedding theorem (see, e.g., \cite[(P. 206, (16)]{Triebel1978}) that
	\[
	\begin{aligned}
		\mathbb{E}\left(\int_{S \wedge \tau}^{T \wedge \tau} f\left(s,X_{s}\right) \mathrm{d} s \middle|\mathcal{F}_{S}\right) & \leq 2 \sup _{s, x}\left|u(s,x)\right|+\sup _{s, x}\left|\nabla u(s,x)\right| \cdot \mathbb{E}\left(\int_{S \wedge \tau}^{T \wedge \tau}\left|\xi_{s}\right| \mathrm{d} s \middle| \mathcal{F}_{S}\right) \\
		& \leq C\|f\|_{L^{q}\left(\mathbb{R}_{+} ; L^{p}(\R^d)\right)}\left(1+\mathbb{E}\left(\int_{S \wedge \tau}^{T \wedge \tau}\left|\xi_{s}\right| \mathrm{d} s \middle| \mathcal{F}_{S}\right)\right),
	\end{aligned}
	\]
	where we have used \( p(\gamma-1)>d \). By a standard density argument, we obtain \eqref{EQ:0715:01} for general \( f \in L^{q}([0, T] \); \( L^{p}(\mathbb{R}^{d})) \).
\end{proof}

\section{Existence of weak solution}\label{SEC:03}

In this section, as an application of Krylov's estimate, we will use the tightness argument and the Skorokhod representation theorem to prove the weak existence for SDE \eqref{EQ:SDE:101} and DDSDE \eqref{EQ:MVSDE:101}. The idea of the proof of the following theorem comes from \cite[Proof of Theorem 4.1]{zhangStochasticDifferentialEquations2013} (see also \cite[Theorem 4.7]{JIN2018}, and \cite[Theorem 2.1]{Huang_SPA_2019} for the case with Gaussian noise).



Specifically, based on Theorem \ref{THM:KRYLOV:01}, coefficients smoothing approximation argument, and compactness methods (Aldous's tightness criterion-Prokhorov's theorem-Skorokhod's representation theorem), and following exactly the approach in Zhang \cite{zhangStochasticDifferentialEquations2013} and related works, we can prove the following result about the existence of weak solution.

\begin{theorem}
Assume \( \alpha \in(1,2],~\gamma \in(1, \alpha),~ p>\frac{d}{\gamma-1} \) and \( q>\frac{\alpha}{\alpha-\gamma} \). For any \( b \in L_{\text{loc}}^{\infty}(\mathbb{R}_{+} ; L^{\infty}(\mathbb{R}^{d}))+ L_{\text{loc}}^{q}(\mathbb{R}_{+} ; L^{p}(\mathbb{R}^{d})) \) and \( x_{0} \in \mathbb{R}^{d} \), there exists a weak solution to SDE \eqref{EQ:SDE:101}. That is, there exists a probability space \( (\tilde{\Omega}, \tilde{\mathcal{F}}, \tilde{\P}) \) and a left-continuous right-limited process \( \tilde{X} \) and \( \tilde{L} \) on it, where \( \tilde{L} \) is a \levy~ process with generate triple $(\sigma,\nu,0)$ and  adapted to the completion of the filtration $(\tilde{\mathcal{F}}_{t})_{t\geq 0}$,  \( \tilde{\mathcal{F}}_{t}:=\sigma^{\tilde{\P}}\{\tilde{X}_{s}, \tilde{L}_{s}, s \leq t\} \), such that
\[
\tilde{X}_{t}=x_{0}+\int_{0}^{t} b(s, \tilde{X}_{s}) \mathrm{d} s+\tilde{L}_{t}, \quad \forall t \geq 0.
\]
\end{theorem}

\begin{remark}
The superposition of independent $\alpha_1$ and $\alpha_2$-stable processes results in a composite process with stability index $\alpha_1 \vee \alpha_2$. Consequently, SDEs driven by such combined noise admit solutions for a broader class of $(p, q)$-parameter pairs. Notably, when the driving noise is a linear combination of two stable processes, the associated integrability constraints are relaxed.
\end{remark}

Next, to prove the weak existence of DDSDE \eqref{EQ:MVSDE:101} using the approximation methods from Mishura and  Veretennikov \cite{Veretennikov2020TPMS}, we need to make some assumptions about its drift coefficient.

Given \( \theta \in [1,\alpha) \) with $\alpha\in (1,2]$,  we will consider SDE with initial distributions in the class
\[
\mathscr{P}_{\theta}:=\left\{\mu \in \mathscr{P}: \mu(|\cdot|^{\theta})<\infty\right\},
\]
where $\mathscr{P}$ is space composed by all the   probability measure among $\R^d.$ 
It is well known that \( \mathscr{P}_{\theta} \) is a Polish space under the Wasserstein distance
\[
\mathbb{W}_{\theta}(\mu, \nu):=\inf _{\pi \in \mathscr{C}(\mu, \nu)}\left(\int_{\mathbb{R}^{d} \times \mathbb{R}^{d}}|x-y|^{\theta} \pi(\mathrm{d} x, \mathrm{d} y)\right)^{{1}/{\theta}}, \quad \forall\mu, \nu \in \mathscr{P}_{\theta},
\]
where \( \mathscr{C}(\mu, \nu) \) is the set of all couplings of \( \mu \) and \( \nu \). Moreover, the topology induced by \( \mathbb{W}_{\theta} \) on \( \mathscr{P}_{\theta} \) coincides with the weak topology. 
Let 
\[ 
\mathscr{P}_{\theta}^{ac}:=\left\{\mu \in \mathscr{P}_{\theta}: \mu \text{ is absolutely continuous with respect to the Lebesgue measure in $\R^d$}\right\}.
\]
There exists a sequence \( (b^{n})_{n \geq 1} \) where
\[
b^{n}:[0, T] \times \mathbb{R}^{d} \times \mathscr{P}_{\theta} \rightarrow \mathbb{R}^{d}
\]
is a measurable function, satisfying the following conditions 
(\hypertarget{Htheta}{\(\mathbf{H}^{\theta}\)}):

\noindent(\hypertarget{(H1)}{H1})
 For \( \mu \in \mathscr{P}_{\theta}^{ac} \) and \( \mu^{n} \rightarrow \mu \) in \( \mathscr{P}_{\theta} \), we have
	\[
	\lim _{n \rightarrow \infty}\left|b^{n}\left(t,x, \mu^{n}\right)-b(t,x, \mu)\right|=0,\quad \text{a.e.}\quad (t, x) \in[0, T] \times \mathbb{R}^{d}.
	\]
	
\noindent(\hypertarget{(H2)}{H2}) Let 
\[
\begin{split}
    L^q([0,T];L^p(\R^d))&:=\bigg\{f\in \mathscr{B}([0,T]\times\R^d)/\mathscr{B}({\R^d}):\int_{0}^{T}\left(\int_{\R^d}|f(t,x)|^p\d x\right)^{q/p}\d t<\infty\bigg\}.
\end{split}
\]
For $\alpha\in (1,2],\gamma\in (1,\alpha),p>\frac{d}{\gamma-1},q>\frac{\alpha}{\alpha-\gamma}.$ There exist constant $K>1,$ and non-negative function $G\in L_{loc}^q(\R_{+};L^p(\R^d)),$ such that for any $n\geq 1,$ 
	\[
		| b^{n}(t,x, \mu)| \leq G(t, x)+K, \quad \forall (t, x, \mu) \in[0, T] \times \mathbb{R}^{d} \times \mathscr{P}_{\theta}.
	\]
	
\noindent(\hypertarget{(H3)}{H3}) For each \( n \geq 1 \), there exists a increasing  function \( K_{n}(\cdot)>0\in C([0,\infty);(0,\infty)) \)  such that \( \left\|b^{n}(t,\cdot)\right\|_{\infty} \leq K_{n}(t) \), and
	\[
		\left|b^{n}(t,x, \mu)-b^{n}(t,y, \nu)\right| \leq K_{n}(t)[|x-y|+\mathbb{W}_{\theta}(\mu, \nu)], \quad \forall (t, x, y) \in[0, T] \times \mathbb{R}^{2d}, \mu, \nu \in \mathscr{P}_{\theta}.
	\]
	

Based on the above assumption (\hyperlink{Htheta}{\(\mathbf{H}^{\theta}\)}), Krylov-type estimates, and Aldous’ tightness criterion, as well as compactness methods, we can construct a weak solution of DDSDE \eqref{EQ:MVSDE:101}. The main result in this part is the following.
\begin{theorem}\label{THM:0915:01}
    Assume that $($\hyperlink{Htheta}{\(\mathbf{H}^{\theta}\)}$)$ hold for some constant $\theta >0.$ Let $X_0$ be an $\mathcal{F}_0$-measurable random variable on $\R^d$ with $\mathscr{L}_{X_0}=\mu_0\in  \mathscr{P}^{\theta}.$ There exists a weak solution to DDSDE \eqref{EQ:MVSDE:101} satisfying $\mathscr{L}_{X_\cdot}=C([0,\infty);\mathscr{P}_{\theta})$.
\end{theorem}
We give the proof of Theorem \ref{THM:0915:01} later.
To this end, we first establish the following preparatory lemmas. The first lemma partially extends Huang and Wang \cite[Lemma 3.4]{Huang_SPA_2019} to the case of general \levy~noise.
\begin{lemma}\label{LEM:0915:01}
   Let \( \left(\bar{\Omega}, (\bar{\mathcal{F}}_{t})_{t\geq 0}, \bar{\mathbb{P}}\right) \) and \( (\bar{X}_{t},\bar{L}_t) \) be a weak solution to DDSDE \eqref{EQ:MVSDE:101} with \( \mu_{t}:=\mathscr{L}_{\bar{X}_{t}}|\bar{\mathbb{P}} \). If the SDE
\begin{equation}\label{EQ:SDE:0912}
    \d X_{t}=b\left(t,X_{t}, \mu_{t}\right) \d t+ \d L_{t},
\end{equation}
has a unique strong solution \( X \) up to life time with \( \mathscr{L}_{X_{0}}=\mu_{0} \), then DDSDE \eqref{EQ:MVSDE:101} has a strong solution.
\end{lemma}

\begin{proof}
    Since \( \mu_{t}=\mathscr{L}_{\bar{X}_{t}}|\overline{\mathbb{P}}, (\bar{X}_{t}) \) is a weak solution to SDE \eqref{EQ:SDE:0912}. By the Yamada-Watanabe principle, the strong uniqueness of SDE \eqref{EQ:SDE:0912} implies the weak uniqueness, so that \( (X_{t}) \) is non-explosive with \( \mathscr{L}_{X_{t}}=\mu_{t}, t \geq 0 \). Therefore, \( (X_{t}) \) is not only the strong solution to SDE \eqref{EQ:SDE:0912} but also DDSDE \eqref{EQ:MVSDE:101}.
\end{proof}

The following lemma establishes existence and uniqueness of strong solutions under Lipschitz type conditions, generalizing results from Huang and Wang \cite[Lemma 3.5]{Huang_SPA_2019} to general \levy~driving noise.

\begin{lemma}\label{LEM:0912:01}
Let $\theta\in[1,\alpha),$ and $\delta_{0}$  be the Dirac measure at point $0.$ If $b(t,0,\delta_0)$ is bounded in $t\in \R_{+},$ and there exists a increasing function  $K(t)\in C([0,\infty);(0,\infty)),$ such that
\begin{equation}\label{EQ:1004:01}
        |b(t,x,\mu)-b(t,y,\nu)|\leq K(t)[|x-y|+W_{\theta}(\mu,\nu)],\quad\forall x,y\in \R^d,\mu,\nu\in \mathscr{P}_{\theta},t>0.
    \end{equation}
    Then for any $X_0$ with $\E[|X_0|^{\theta}]<\infty,$ DDSDE \eqref{EQ:MVSDE:101} has a unique strong solution $(X_t)_{t\geq 0 }.$
 \end{lemma}
\begin{proof}
    For each $n\geq 1,$ we have the following iterated SDE
 \begin{equation}\label{EQ:1004:02}
         \d X_t^{(n)}= b(t,X_t^{(n)},\mu_{t}^{(n-1)}) \d t+\d L_t, 
 \end{equation}
    where $\mu_{t}^{(n-1)}=\mathscr{L}_{X_t^{(n-1)}}.$ Note that 
    \begin{equation*}
        \d L_t =\sigma \d B_{t}+\int_{|x|\leq 1} x\tilde{N}(\d t,\d x)+\int_{|x|\geq 1} x{N}(\d t,\d x),
    \end{equation*}
    then by using \ito's formula to $(1+|X_t^{(n)}|^2)^{\theta/2},$ one has
\begin{align*}
    &\quad \d (1+|X_t^{(n)}|^2)^{\frac{\theta}{2}} \\&=\frac{\theta}{2} (1+|X_t^{(n)}|^2)^{\frac{\theta}{2}-1}\left(2\langle X_t^{(n)},b(t,X_t^{(n)},\mu_t^{(0)})\rangle+\operatorname{tr}(\sigma\sigma^{\top})+\frac{(\theta-2)|\sigma^{\top}X_t^{(n)}|^2}{1+|X_t^{(n)}|^2}\right)\d t\\
    &\quad+\theta(1+|X_t^{(n)}|^2)^{\frac{\theta}{2}-1}\langle X_t^{(n)},\sigma\d W_t\rangle+ \int_{|x|\leq 1} \left[\left(1+|X_t^{(n)}+x|^2\right)^{\frac{\theta}{2}}-\left(1+|X_t^{(n)}|^2\right)^{\frac{\theta}{2}}\right] \tilde{N}(\d t,\d x)\\
    &\quad+ \int_{|x|\leq 1} \left[\left(1+|X_t^{(n)}+x|^2\right)^{\frac{\theta}{2}}-\left(1+|X_t^{(n)}|^2\right)^{\frac{\theta}{2}} - \left\langle x, \theta(1+|X_t^{(n)}|^2)^{\frac{\theta}{2}-1}X_t^{(n)}\right\rangle\right] \nu(\d x)\d t\\
    &\quad +\int_{|x|> 1} \left[\left(1+|X_t^{(n)}+x|^2\right)^{\frac{\theta}{2}}-\left(1+|X_t^{(n)}|^2\right)^{\frac{\theta}{2}}\right] N(\d t,\d x).
\end{align*}
Theorem 2.1 in \cite{butkovskyStrongRateConvergence2024} guarantees that SDE \eqref{EQ:1004:02} admits a unique strong solution $X^{(n)}$. Furthermore, following the argument as the proof of Lemma 2.3-(1) in \cite{Wang_SPA_2018}, we  have
    $\sup_{s\leq T} \E[|X_s^{(n)}|^{\theta}]<\infty.$ 

    Let $\xi_t^{(n)}:=X_t^{(n+1)}-X_t^{(n)}.$ Note that
    \[
    \d \xi_t^{(n)}=[b(t,X_t^{(n+1)},\mu_{t}^{(n)})-b(t,X_t^{(n)},\mu_{t}^{(n-1)})]\d t.
    \]
    By Newton-Leibniz formula and \eqref{EQ:1004:01}, there exists increasing function $K_0(\cdot):[0,\infty)\to [0,\infty),$ 
   \begin{equation}\label{EQ:0407:03}
       \begin{split}
             \begin{split}
        \d [|\xi_t^{(n)}|^2] & =
        2\langle \xi_t^{(n)}, b(t,X_t^{(n+1)},\mu_t^{(n)})-b(t,X_t^{(n)},\mu_t^{(n-1)})\rangle \d t\\
        &\leq 2|\xi_t^{(n)}||b(t,X_t^{(n+1)},\mu_{t}^{(n)})-b(t,X_t^{(n)},\mu_{t}^{(n-1)})| \d t\\
        & \leq K_0(t)[|\xi_t^{(n)}|^2+\mathbb{W}_{\theta}(\mu_t^{(n)},\mu_t^{(n-1)})^2]\d t.
   \end{split}
       \end{split}
   \end{equation}
   Recall that we suppose that $\theta\in [1,\alpha).$ Since \( \xi_{0}^{(n)}=0 \), it follows from \eqref{EQ:0407:03}  and \gron's lemma that
\[
\begin{aligned}
\mathbb{E}[|\xi_{t}^{(n)}|^{2}] & \leq \int_{0}^{t} K_{0}(s) \mathbb{W}_{\theta}\left(\mu_{s}^{(n)}, \mu_{s}^{(n-1)}\right)^{2} \mathrm{d} s \left(\exp{\int_{0}^{t}K_0(s)\d s}\right) \\
& \leq \int_{0}^{t}K_{0}(s)\d s  \left(\exp{\int_{0}^{t}K_{0}(s)\d s}\right) \sup _{s \in[0, t]}\left(\mathbb{E}\left|\xi_{s}^{(n-1)}\right|^{\theta}\right)^{\frac{2}{\theta}}\\
&\leq K_0(T)t\e^{K_0(T)T}\sup _{s \in[0, t]}\left(\mathbb{E}\left[\left|\xi_{s}^{(n-1)}\right|^{\theta}\right]\right)^{\frac{2}{\theta}}, \quad t \in[0, T], n \geq 1.
\end{aligned}
\]
By Jensen's inequality, we can find a constant \( K_{1}>0 \) such that
\[
\sup _{s \in[0, t]} \mathbb{E}[|\xi_{s}^{(n)}|^{\theta}] \leq K_{1} t^{\frac{\theta}{2}} \sup _{s \in[0, t]} \mathbb{E}\left[|\xi_{s}^{(n-1)}|^{\theta}\right], \quad n \geq 1, t \in[0, T].
\]
So, taking \( t_{0} \in(0, T \wedge K_{1}^{-{2}/{\theta}}) \), we can find a constant \( \varepsilon \in(0,1) \) such that
\[
\sup _{s \in\left[0, t_{0}\right]} \mathbb{E}[|\xi_{s}^{(n)}|^{\theta}] \leq \varepsilon^{n} \sup _{s \in\left[0, t_{0}\right]} \mathbb{E}[|X_{s}^{(1)}-X_{0}|^{\theta}], \quad n \geq 1.
\]
Therefore, $(\xi_{t}^{(n)})_{n\geq 1} =(X_t^{(n+1)}-X_t^{(n)})_{n\geq 1}$ is Cauchy sequence in $L^{\theta}(\R^d)$ for any \( t \in\left[0, t_{0}\right] \). Moreover, there exists an \( \mathcal{F}_{t} \)-measurable random variable \( X_{t} \) on \( \mathbb{R}^{d} \) such that
\[
\lim _{n \rightarrow \infty} \sup _{t \in\left[0, t_{0}\right]} \mathbb{W}_{\theta}(\mu_{t}^{(n)}, \mu_{t})^{\theta} \leq \lim _{n \rightarrow \infty} \sup _{t \in[0, t_{0}]} \mathbb{E}[|X_{t}^{(n)}-X_{t}|^{\theta}]=0,
\]
where \( \mu_{t}:=\mathscr{L}_{X_{t}} \). Combining this with \eqref{EQ:1004:01} and letting \( n \rightarrow \infty \) in the equation
\[
X_{t}^{(n)}=\int_{0}^{t} b(s,X_{s}^{(n)}, \mu_{s}^{(n-1)}) \mathrm{d} s+L_t, \quad n \geq 1, t \in\left[0, t_{0}\right].
\]
we derive for every \( t \in\left[0, t_{0}\right],\)
\[
X_{t}=\int_{0}^{t} b(s,X_{s}, \mu_{s}) \mathrm{d} s+L_t.
\]
Thus, \((X_{s})_{s \in[0, t_{0}]} \) has a version with right continuous and left limits that is a strong solution of \eqref{EQ:MVSDE:101} up to time \( t_{0} \). The uniqueness is trivial by using condition \eqref{EQ:1004:01} and  \ito's formula. Finally, since $t_{0}>0$ is independent of $X_{0}$ and $\sup _{t \in\left[0, t_{0}\right]} \mathbb{E}[\left|X_{t}\right|^{\theta}]<\infty$, we conclude that \eqref{EQ:MVSDE:101} has a unique solution $\left(X_{t}\right)_{t \geq 0}$.
The proof is complete.
\end{proof}
Based on the preparations outlined above, we can now proceed to prove the main result by means of distribution iteration argument. 
\begin{proof}[Proof of Theorem \ref{THM:0915:01}]
By Lemma \ref{LEM:0912:01}, the condition (\hyperlink{H3}{H3}) in (\hyperlink{Htheta}{\(\mathbf{H}^{\theta}\)}) implies that for any \(n \geq 1\), the approximation DDSDE
    \begin{equation}\label{EQ:0408:01}
        \d X_t^n =b^n(t,X_t^n,\mathscr{L}_{X_t^n})\d t+\d L_t,\quad X_0^n=X_0
    \end{equation}
    has a unique strong solution $(X_t^n)_{t\geq 0}.$

    As usual, let \( \mathbb{D} \) be the space of all \( \mathbb{R}^{d} \)-valued càdlàg functions on \( \R_{+} \) equipped with the Skorokhod topology such that \( \mathbb{D} \) is a Polish space. We establish the following claims, respectively.

    Claim $1$, For some $\delta >1,$ we have 
    \begin{equation}
        \sup_{n\in \mathbb{N}}\E\int_{0}^{T}\left|b^n(s,X_s^n,\mathscr{L}_{X_s^n})\right|^{\delta} \d s <\infty,\quad \forall T>0.
    \end{equation}
    In fact, choosing $\delta>1,$ and $p^{\prime}\in (\frac{d}{\gamma-1},p),q^{\prime}\in (\frac{\alpha}{\alpha-\gamma},q)$ such that $p^{\prime}\delta=p,q^{\prime}\delta=q,$ by (\hyperlink{(H2)}{H2}) in $($\hyperlink{{Htheta}}{\(\mathbf{H}^{\theta}\)}$)$,
    \begin{equation*}
        \begin{split}           \E\int_{0}^{T}\left|b^n(s,X_s^n,\mathscr{L}_{X_s^n})\right|^{\delta} \d s
        &\leq \E \int_{0}^{T}|G(s,X_s^n)|^{\delta}\d s,
        \end{split}
    \end{equation*}
   and  moreover, reformulate DDSDE \eqref{EQ:0408:01} as 
   \begin{equation*}
       X_t^n =X_0+\int_{0}^{t}b^n(s,X_s^n,\mathscr{L}_{X_t^n})\d s+L_t =:X_0+\int_{0}^{t}\xi_s\d s+L_t,
   \end{equation*}
    then due to Theorem \ref{THM:KRYLOV:01}, and Young's inequality, one has
    \begin{equation*}
        \begin{split}
             \E \int_{0}^{T}|G(s,{X}_s^n)|^{\delta}\d s 
        &\leq C_T\left(1+\E \int_{0}^{T}|G(s,{X}_s^n)|\d s\right)\||G|^{\delta}\|_{L^{q^{\prime}}([0,T];L^{p^{\prime}}(\R^d))}\\         
        &\leq C_T\left(1+\E \int_{0}^{T}|G(s,{X}_s^n)|\d s\right)\cdot\|G\|^{\delta}_{L^q([0,T];L^{p}(\R^d))}\\
        &\leq \frac{1}2\E \int_{0}^{T}|G(s,X_s^n)|^{\delta}\d s+ C_T\|G\|^{\frac{\delta^2}{\delta-1}}_{L^q([0,T];L^{p}(\R^d))}\\
        &\quad+C_T\|G\|^{\delta}_{L^q([0,T];L^{p}(\R^d))}<\infty.
        \end{split}
    \end{equation*}
    So, there exists $\delta>1,$ such that for any $T>0,$ 
    \[
    \sup_{n\in \mathbb{N}}\E\int_{0}^{T}\left|b^n(s,X_s^n,\mathscr{L}_{X_s^n})\right|^{\delta} \d s <\infty.
    \]
Then, set $H_t^n:=\int_{0}^{t}b^n(s,X_s^n,\mathscr{L}_{X_s^n})\d s,$ using Claim $1,$ it is easy to check that the following Aldous's tightness criterions (cf. \cite{Aldous1978AoP}) hold:
\[
\lim _{N \rightarrow \infty} \varlimsup_{n \rightarrow \infty} \P\bigg(\sup _{t \in[0, T]}\left|H_{t}^{n}\right| \geq N\bigg)=0, \quad \forall T>0,
\]
and
\[
\lim _{\varepsilon \rightarrow 0} \varlimsup_{n \rightarrow \infty} \sup _{\tau \in \mathcal{S}_{T}} \P\left(\left|H_{\tau}^{n}-H_{\tau+\varepsilon}^{n}\right| \geq a\right)=0, \quad \forall T, a>0,
\]
where \( \mathcal{S}_{T} \) denotes all the bounded stopping times with bound \( T \). Thus, the law of \( t \mapsto H_{t}^{n} \) in \( \mathbb{D} \) is tight, and so is \( (H^{n}, L \)). By Prokhorov's theorem, there exists a subsequence still denoted by \( \{n\} \) such that the law of \( (H^{n}, L) \) in \( \mathbb{D} \times\mathbb{D} \) weakly converges, which then implies that the law of \( (X^{n}, L\)) weakly converges. By Skorokhod's representation theorem, there is a probability space \( (\tilde{\Omega}, \tilde{\mathcal{F}}, \tilde{\P}) \) and the \( \mathbb{D} \times \mathbb{D} \)-valued random variables \( (\tilde{X}^{n}, \tilde{L}^{n}) \) and \( (\tilde{X}, \tilde{L})\) such that

(i) $(\tilde{X}^{n}, \tilde{L}^{n})$  has the same law as $(X^{n}, L) $ in $ \mathbb{D} \times \mathbb{D}.$

(ii) \( (\tilde{X}^{n}, \tilde{L}^{n}) \) converges to \( (\tilde{X},\tilde{L}\)),~\( \tilde{\P}\)-a.s.

    In particular, $\tilde{L}$ is still a \levy~process with respect to the completed filtration   $\tilde{\mathcal{F}}_t:=\sigma^{\tilde{\P}}(\tilde{X}_{s}, \tilde{L}_{s}, s \leq t)$ and  has the same \levy~symbol as process $L,$ moreover
    \begin{equation}
        \tilde{X}_{t}^{n}=\tilde{X}_{0}+\int_{0}^{t} b^{n}(s, \tilde{X}_{s}^{n},\mathscr{L}_{\tilde{X}_{s}^{n}}) \mathrm{d} s+\tilde{L}_{t}^{n},
    \end{equation}
    with $\mathscr{L}_{\tilde{X}_{0}^{n}}|\tilde{\P}=\mathscr{L}_{{X}_{0}^{n}}|{\P}.$

    Claim 2: For any non-negative measurable function \( f \in L^q([0,T];L^{p}(\R^d))\) and \( T>0 \), we have
\[
\tilde{\mathbb{E}} \int_{0}^{T} f(s,\tilde{X}_{s}) \mathrm{d} s \leq C_{T}\|f\|_{L^{q}([0, T] ; L^{p}(\mathbb{R}^{d}))},
\]
where \( \tilde{\mathbb{E}} \) denotes the expectation with respect to the probability measure \( \tilde{\P} \).
Let \( f \in \bC_{0}([0, T] \times \mathbb{R}^{d}) \). By the dominated convergence theorem, we have
\begin{equation}\label{EQ:0912:02}
    \tilde{\mathbb{E}} \int_{0}^{T} f(s,\tilde{X}_{s}) \mathrm{d} s=\lim _{n \rightarrow \infty} \tilde{\mathbb{E}} \int_{0}^{T} f(s,\tilde{X}_{s}^{n}) \mathrm{d} s=\lim _{n \rightarrow \infty} \mathbb{E} \int_{0}^{T} f(s,X_{s}^{n}) \mathrm{d} s \leq C\|f\|_{L^{q}([0, T] ; L^{p}(\mathbb{R}^{d}))}
\end{equation}
where in the last step we have used  (\hyperlink{H2}{H2}) in $($\hyperlink{{Htheta}}{\(\mathbf{H}^{\theta}\)}$)$ and Krylov type estimate \eqref{EQ:0715:01}. For general \( f \), it follows by a standard monotone class argument. 

The proof will be finished if one can show the following claim.

Claim 3: For any  $T>0$, we have
\begin{equation}
    \lim _{n \rightarrow \infty} \tilde{\mathbb{E}}\bigg(\int_{0}^{T}\left|b^{n}(s, \tilde{X}_{s}^{n},\mathscr{L}_{\tilde{X}_{s}^{n}})-b(s, \tilde{X}_{s},\mathscr{L}_{\tilde{X}_{s}})\right| \mathrm{d} s\bigg)=0.
\end{equation}
 For any $m,n\in \mathbb{N},$ 
\begin{equation}
     \begin{aligned}
&\quad \tilde{\mathbb{E}}\left(\int_{0}^{t}\left|b^{n}(s, \tilde{X}_{s}^{n},\mathscr{L}_{\tilde{X}_{s}^{n}})-b(s, \tilde{X}_{s},\mathscr{L}_{\tilde{X}_{s}})\right| \mathrm{d} s\right)\\
&\leq  \tilde{\mathbb{E}}\left(\int_{0}^{t}\left|b^{n}(s, \tilde{X}_{s}^{n},\mathscr{L}_{\tilde{X}_{s}^{n}})-b^{m}(s, \tilde{X}_{s}^{n},\mathscr{L}_{\tilde{X}_{s}})\right| \mathrm{d} s\right) \\
& \quad +\tilde{\mathbb{E}}\left(\int_{0}^{t}\left|b^{m}(s, \tilde{X}_{s}^{n},\mathscr{L}_{\tilde{X}_{s}})-b^{m}(s, \tilde{X}_{s},\mathscr{L}_{\tilde{X}_{s}})\right| \mathrm{d} s\right) \\
& \quad +\tilde{\mathbb{E}}\left(\int_{0}^{t}\left|b^{m}(s, \tilde{X}_{s},\mathscr{L}_{\tilde{X}_{s}})-b(s, \tilde{X}_{s},\mathscr{L}_{\tilde{X}_{s}})\right| \mathrm{d} s\right) \\
&= : I_{1}^{(n, m)}(t)+I_{2}^{(n, m)}(t)+I_{3}^{(n, m)}(t) .
\end{aligned}
\end{equation}
We estimate $I_{i}^{(n, m)},i=1,2,3,$ respectively.  For simplicity, let $\tilde{\mu}_{t}=\mathscr{L}_{\tilde{X}_{t}} $ and $\tilde{\mu}_{t}^{n}=\mathscr{L}_{\tilde{X}_{t}^{n}}$, 
    \begin{align*}
        I_{1}^{(n, m)}(t) &=\E \left(\int_{0}^{t} \1_{\{|\tilde{X}_s^n|\leq R\}}\left|b^n(s,\tilde{X}_s^n,\tilde{\mu}_s^n)-b^m(s,\tilde{X}_s^n,\tilde{\mu}_s)\right|\d s\right)\\
        &\quad +\E \left(\int_{0}^{t} \1_{\{|\tilde{X}_s^n|> R\}}\left|b^n(s,\tilde{X}_s^n,\tilde{\mu}_s^n)-b^m(s,\tilde{X}_s^n,\tilde{\mu}_s)\right|\d s\right)\\
        &\leq C\left[\int_{0}^{t}\left(\int_{|x|\leq R}\left|b^n(s,x,\tilde{\mu}_s^n)-b^m(s,x,\tilde{\mu}_s)\right|^p \d x\right)^{\frac{q}{p}}\right]^{\frac{1}{q}}\\
        &\quad+K\int_{0}^{t}\tilde{\P}(|\tilde{X}_s^n|>R)\d s+C\|G\1_{\{|\cdot|>R\}}\|_{L^q([0,T];L^p(\R^d))}.
    \end{align*}
Since $\tilde{X}_t^n\overset{\tilde{\P}\text{-a.s.}}{\longrightarrow} \tilde{X}_t,$ so  $\tilde{X}_t^n\overset{\tilde{\P}}{\longrightarrow} \tilde{X}_t,$ it is clear that 
\begin{equation}
    \lim_{n\to\infty}\mathbb{W}_{\theta}(\tilde{\mu}_t^n,\mu_t)=0,
\end{equation}
and 
\begin{equation}
    \lim_{n\to\infty}\tilde{\P}(|\tilde{X}_t^n|>R)\leq \tilde{\P}(|\tilde{X}_t|\geq {R}/{2}).
\end{equation}
Then it follows (\hyperlink{(H1)}{H1}), (\hyperlink{(H3)}{H3}) in $($\hyperlink{Htheta}{\(\mathbf{H}^{\theta}\)}$)$ that 
\begin{equation}
\lim_{n\to\infty}b^n(t,x,\tilde{\mu}_t^n)-b(t,x,\tilde{\mu}_t)=0,\quad \text{a.e.} \quad t\geq 0,x\in\R^d.
\end{equation}
So by (\hyperlink{H2}{H2}) and the dominated convergence theorem,
\begin{equation}\label{EQ:0912:01}
    \begin{split}
        \lim_{n\to\infty}I_1^{(n,m)}(t)& \leq C\left[\int_{0}^{t}\left(\int_{|x|\leq R}|b(s,x,\tilde{\mu}_s)-b^m(s,x,\tilde{\mu}_s)|^p\d x\right)^{\frac{q}{p}}\right]^{\frac{1}{q}}\\
        &\quad+K\int_{0}^{t}\tilde{\P}(|\tilde{X}_s|>R)\d s+C\|G\1_{\{|\cdot|>R\}}\|_{L^q([0,T];L^p(\R^d))}.
    \end{split}
\end{equation}

Similarly, due to $\tilde{X}_t^n\overset{\tilde{\P}}{\longrightarrow} \tilde{X}_t$ and Claim 2, we have that \eqref{EQ:0912:01} holds for $I_3^{(n,m)}$ replacing $I_1^{(n,m)}$. 

Moreover, since $b^m$ is bounded and continuous, it follows from item (ii) above and the dominated convergence theorem that 
\begin{equation}
    \varlimsup_{n\to\infty}I_2^{(n,m)}(t) =0.
\end{equation}

In conclusion, we arrive at
\begin{equation}
\begin{split}
       & \quad \varlimsup_{n\to\infty}\E\int_{0}^{t}\left|b^n(s,\tilde{X}_s^n,\mathscr{L}_{\tilde{X}_s^n})-b(s,\tilde{X}_s,\mathscr{L}_{\tilde{X}_s})\right|\d s\\
    &\leq 2C\left[\int_{0}^{t}\left(\int_{|x|\leq R}\left|b(s,x,\tilde{\mu}_s)-b^m(s,x,\tilde{\mu}_s)\right|^p\d x\right)^{\frac{q}{p}}\right]^{\frac{1}{q}}\\
        &\quad+2K\int_{0}^{t}\tilde{\P}(|\tilde{X}_s|>R)\d s+2C\|G\1_{\{|\cdot|>R\}}\|_{L^q([0,T];L^p(\R^d))},
\end{split}
\end{equation}
for any $m>0,$ and $R>0.$ Then letting $m\to\infty,$ and $R\to\infty,$ due to the condition (\hyperlink{H1}{H1}), (\hyperlink{H2}{H2}) in $($\hyperlink{Htheta}{\(\mathbf{H}^{\theta}\)}$)$, we obtain from the dominated convergence theorem that 
\begin{equation*}
    \varlimsup_{n\to\infty} \tilde{\E}\left(\int_{0}^{T}\left|b^n(s,\tilde{X}_s^n,\mathscr{L}_{\tilde{X}_s^n})-b(s,\tilde{X}_s,\mathscr{L}_{\tilde{X}_s})\right|\d s\right)=0.
\end{equation*}
In summary, $(\tilde{X},\tilde{L})$ is a weak solution to DDSDE \eqref{EQ:MVSDE:101} with initial distribution $\mu_0.$
\end{proof}


\begin{remark}
    According to \cite[Theorem 1.1]{zhangStochasticDifferentialEquations2013}, when $L$ satisfies Assumption $($\hyperlink{(A)}{$\mathbf{A}$}$)$ and moreover, is exactly a $\alpha$-stable process, then  the condition $($\hyperlink{Htheta}{\(\mathbf{H}^{\theta}\)}$)$ implies that the SDE
    \begin{equation}
        \d X_t =b(t,X_t,\mu_t)\d t+\d L_t,\quad t\geq 0,
    \end{equation}
   where $\mu_t$ is given in \eqref{EQ:SDE:0912}, has a unique strong solution. Thus, by Lemma \ref{LEM:0915:01}, the strong existence of DDSDE \eqref{EQ:MVSDE:101} can follows from the weak existence result  Theorem \ref{THM:0915:01}. 
\end{remark}

\bibliographystyle{plain}

\bibliography{reference_DDSDE_Levy}
\end{document}